\documentclass[12pt]{amsart}
\usepackage{amsmath,amsthm,amsfonts,amssymb,latexsym}
\usepackage{enumerate}

\newcommand{\ZZ}{{\mathbb Z}}
\newcommand{\QQ}{{\mathbb Q}}
\newcommand{\FF}{\mathbb{F}}

\newcommand{\bG}{{\mathbf{G}}}
\newcommand{\bT}{{\mathbf{T}}}
\newcommand{\bN}{{\mathbf N}}
\newcommand{\bZ}{{\mathbf Z}}
\newcommand{\bC}{{\mathbf C}}

\newcommand{\AAA}{{\sf A}}
\newcommand{\SSS}{{\sf S}}

\newcommand{\Cl}{\operatorname{Cl}}
\newcommand{\GL}{\operatorname{GL}}
\newcommand{\SL}{\operatorname{SL}}
\newcommand{\PSL}{\operatorname{PSL}}
\newcommand{\GU}{\operatorname{GU}}
\newcommand{\SU}{\operatorname{SU}}
\newcommand{\PSU}{\operatorname{PSU}}
\newcommand{\SO}{\operatorname{SO}}
\newcommand{\Sp}{\operatorname{Sp}}
\newcommand{\PSp}{\operatorname{PSp}}

\newcommand{\Hom}{{\operatorname{Hom}}}
\renewcommand{\Im}{{\operatorname{Im}}}
\newcommand{\Ker}{\operatorname{Ker}}
\newcommand{\ord}{{\operatorname{ord}}}
\newcommand{\diag}{{\operatorname{diag}}}

\def\cent#1#2{{\bf C}_{#1}(#2)}
\def\syl#1#2{{\rm Syl}_#1(#2)}
\def\nor{\triangleleft\,}
\def\oh#1#2{{\bf O}_{#1}(#2)}
\def\norm#1#2{{\bf N}_{#1}(#2)}
\newcommand{\fv}{\check{f}}

\newcommand{\tw}[1]{{}^#1\!}
\let\eps=\epsilon

\theoremstyle{plain}

\newtheorem{thm}{Theorem}[section]
\newtheorem{lem}[thm]{Lemma}
\newtheorem{prop}[thm]{Proposition}
\newtheorem{cor}[thm]{Corollary}

\newtheorem*{thmA}{Theorem A}
\newtheorem*{thmB}{Theorem B}
\newtheorem*{thmC}{Theorem C}

\begin{document}

\marginparsep-0.5cm
\footnotesep6.5pt

\title[Nilpotent and Abelian Hall Subgroups]{Nilpotent and Abelian Hall Subgroups\\ in Finite Groups}
\author[Beltr\'an]{Antonio Beltr\'an}
\address{Departamento de Matem\'aticas, Universidad Jaume I, 12071 Castell\'on, Spain}
\email{abeltran@mat.uji.es}
\author[Felipe]{Mar\'ia Jos\'e Felipe}
\address{Instituto Universitario de Matem\'atica Pura y Aplicada, Universidad Polit\'ecnica de Valencia, 46022 Valencia, Spain}
\email{mfelipe@mat.upv.es}
\author[Malle]{Gunter Malle}
\address{FB Mathematik, TU Kaiserslautern, Postfach 3049,
  67653 Kaisers\-lautern, Germany}
\email{malle@mathematik.uni-kl.de}
\author[Moret\'o, Navarro, Sanus]{Alexander Moret\'o, Gabriel Navarro,
Lucia Sanus}
\address{Departament d'\`Algebra, Universitat de Val\`encia, 46100 Burjassot,
Val\`encia, Spain}
\email{alexander.moreto@uv.es, gabriel.navarro@uv.es, lucia.sanus@uv.es}
\author[Tiep]{Pham Huu Tiep}
\address{Department of Mathematics, University of Arizona, Tucson, AZ 85721,
USA}
\email{tiep@math.arizona.edu}

\thanks{
The research of Beltr\'an, Felipe, Moret\'o, Navarro, and Sanus is supported by
the Prometeo/Generalitat Valenciana, Proyectos MTM2010-15296,
MTM2010-19938-C03-02 Fundacio Bancaixa P11B2010-47 and Fondos Feder.
Malle gratefully
acknowledges financial support by ERC Advanced Grant 291512. Tiep gratefully
acknowledges the support of the NSF (grants DMS-0901241 and DMS-1201374).}


\begin{abstract}
We give a characterization of the finite groups having nilpotent or
abelian Hall $\pi$-subgroups which can easily be verified from the character
table.
\end{abstract}

\maketitle

\section{Introduction}

One of the main themes in Finite Group Theory is to study the interaction
between global and local structure: if $p$ is a prime and $G$ is a finite
group, we seek to analyze the relationship between $G$ and its $p$-local
subgroups. Ideally, a local property of $G$ can be read off from the character
table of $G$.

\medskip

There are not many theorems analyzing local and global structure from the point
of view of two different primes, as we do in the main result of this paper.

\begin{thmA}
 Let $G$ be a finite group, and let $p$ and $q$ be different primes.
 Then some Sylow $p$-subgroup of $G$ commutes with some Sylow $q$-subgroup of
 $G$ if and only if the class sizes of the $q$-elements of $G$ are not
 divisible by $p$ and the class sizes of the $p$-elements of $G$ are not
 divisible by $q$.
\end{thmA}

Theorem A, of course, gives us a characterization (detectable in the character
table) of when a finite group possesses nilpotent Hall $\{p,q\}$-subgroups.
The existence of Hall subgroups (nilpotent or not) is a classical subject in
finite group theory, with extensive literature. This characterization in our
Theorem A can be seen as a contribution to Richard Brauer's Problem 11.
In his celebrated paper \cite{Br}, Brauer asks about obtaining information
about the existence of (certain) subgroups of a finite group given its
character table.

\medskip
In the course of proving Theorem~A we show that for finite simple groups,
commuting Sylow subgroups for different primes are actually always abelian,
see Theorem~\ref{simple}.

\medskip
Several possible extensions of Theorem A are simply not true. For instance,
the fact that $p$ does not divide $|G:\cent Gx|$ for every $q$-element
$x \in G$ does not guarantee that a Sylow $p$-subgroup of $G$ commutes with
some Sylow $q$-subgroup of $G$: the semi-affine groups of order $q^p(q^p-1)p$
provide solvable examples for every choice of $p$ and $q$.
Furthermore, if $p$ does not divide $|G:\cent Gx|$ for every $q$-element
$x \in G$, it is not even true that $P$ normalizes some Sylow $q$-subgroup of
$G$ ($G=M_{23}$ for $(q,p)=(3,5)$, or $G=J_4$ with $(q,p)=(3,7)$ are examples).
However, this assertion is true for $p$-solvable or $q$-solvable groups, or
if $p=2$, as we will show in Section 4 below.
\medskip

Recently, the finite groups with nilpotent Hall subgroups have also received
attention in \cite{M}. Using Theorem A and the results of \cite{M},
we can deduce the following.

\begin{thmB}
 Let $G$ be a finite group, and let $\pi$ be a set of primes.
 Then $G$ has nilpotent Hall $\pi$-subgroups if and only if
 for every pair of distinct primes $p, q\in \pi$, the class sizes of the
 $p$-elements of $G$ are not divisible by $q$.
\end{thmB}

Theorem B gives an easy algorithm to determine if a group has nilpotent Hall
$\pi$-subgroups from the character table. (The somewhat weaker result that
the property of having nilpotent Hall subgroups is shared by groups with the
same character table was proven in \cite{KS}.)

\medskip

As a consequence of Theorem B and the main results of \cite{NT} and \cite{NST}
now we have the following explicit way to detect from the character table of
a finite group $G$ whether $G$ possesses an abelian Hall $\pi$-subgroup, for
any set $\pi$ of primes.

\begin{thmC}
 Let $G$ be a finite group, and let $\pi$ be a set of primes.
 Then $G$ has abelian Hall $\pi$-subgroups if and only if the two following
 conditions hold:
 \begin{enumerate}[\rm(i)]
  \item For every $p \in \pi$ and every $p$-element $x \in G$, $|G:\cent Gx|$
   is a $\pi'$-integer.
  \item For every $p \in \pi \cap \{3,5\}$ and for every irreducible character
   $\chi$ in the principal $p$-block of $G$, $\chi(1)$ is not divisible by $p$.
 \end{enumerate}
\end{thmC}

As it might be expected, the proofs of Theorems A, B and C use the
Classification of Finite Simple Groups.

\section{Simple Groups and Theorem A}
The goal of this section is to prove Theorem \ref{simple}, which
yields a strong form of Theorem A for simple groups, and another
auxiliary result, Theorem \ref{simple2}.

\begin{thm}   \label{simple}
 Let $S$ be a finite non-abelian simple group and let $p$ and $q$ be distinct
 prime divisors of $|S|$.
 Suppose that all the $p$-elements of $S$ have $q'$-conjugacy class
 sizes and all the $q$-elements of $S$ have $p'$-conjugacy class sizes. Then
 $p, q > 2$ and $S$ has an abelian Hall $\{p,q\}$-subgroup.
\end{thm}

Since, in contrast to Theorem~A, we claim existence of abelian Hall subgroups,
Theorem \ref{simple} obviously does not generalize to arbitrary finite groups.

\begin{thm}   \label{simple2}
 Let $S$ be a finite non-abelian simple group of order divisible by an odd
 prime $r$. Then $S$ contains a conjugacy class of $r$-elements of even size.
\end{thm}

We begin with some obvious observations. If $x \in G$, then we denote
by $x^G$ the conjugacy class of $x$ in $G$.

\begin{lem}  \label{center}
 \begin{enumerate}[\rm(i)]
  \item For each simple group $S$ it suffices to prove Theorem~\ref{simple}
   and Theorem~\ref{simple2} for some quasi-simple group $L$ such that
   $S \cong L/\bZ(L)$.
  \item The case $\min(p,q) = 2$ of Theorem~\ref{simple} follows from
   Theorem \ref{simple2}.
 \end{enumerate}
\end{lem}

\begin{proof}
(i) Suppose that $S$ satisfies the hypothesis of Theorem \ref{simple}, and
that Theorem~\ref{simple} holds for some quasi-simple group $L$ with
$S = L/\bZ(L)$. Let $g \in L$ be any $p$-element and let $D/Z := \bC_S(gZ)$
for $Z := \bZ(L)$. Then for any $x \in D$ we have $xgx^{-1} = f(x)g$ for some
$f(x) \in Z$. Moreover, $f \in \Hom(D,Z)$, $\Ker(f) = \bC_L(g)$, and
$f(x)^{|g|} = 1$ for all $x \in G$, i.e. $\Im(f)$ is a $p$-group.
Thus $|D:C|$ is a $p$-power. It follows that $|(gZ)^S| = |L:D|$ and
$|g^L| = |L:C|$ have the same $q$-part and so $|g^L|$ is coprime to $q$.
Similarly, $|h^L|$ is coprime to $p$ for all $q$-elements $h \in L$. Since
Theorem~\ref{simple} holds for $L$, $L$ contains an abelian Hall
$\{p,q\}$-subgroup $P \times Q$, whence $(P \times Q)Z/Z$ is an abelian Hall
$\{p,q\}$-subgroup for $S$.

A similar argument proves the part of the claim concerning Theorem
\ref{simple2}.

\smallskip
(ii) Suppose $p = 2 < q$. Since $q||S|$, by Theorem \ref{simple2}
there is a conjugacy class $x^S$ of $q$-elements in $S$ of even size,
a contradiction.
\end{proof}

\begin{lem}   \label{alt}
 Theorems~\ref{simple} and~\ref{simple2} hold in the case $S$ is an
 alternating group, a sporadic group, or $\tw2 F_4(2)'$.
\end{lem}

\begin{proof}
The case of $26$ sporadic groups and $\tw2 F_4(2)'$ can be checked directly
using \cite{GAP}. (We remark that the only examples among the sporadic groups
are $J_1$ for $\{p,q\} =\{3,5\}$ and $J_4$ for $\{p,q\} = \{5,7\}$, and in both
cases $S$ has cyclic Hall $\{p,q\}$-subgroups.) Suppose that $S = \AAA_n$ and
$n \geq p > q \geq 2$. Then $\bC_{\SSS_n}(g) \cong C_p \times \SSS_{n-p}$ for
a $p$-cycle $g \in S$. If $p \geq 5$, we have that
$|g^{\SSS_n}| = \binom{n}{p} (p-1)!$ is divisible by $2q$ and so
$q$ divides $|g^S|$. If $p = 3$, then $n \geq p+2$ and so
$|g^S| = |g^{\SSS_n}|$ is again divisible by $q$.
\end{proof}

The rest of the section is devoted to proving Theorems \ref{simple}
and \ref{simple2} for simple groups of Lie type $S \not\cong \tw2 F_4(2)'$.
For the sake of convenience, we rename $p,q$ in Theorem \ref{simple} to $r,s$.
We will consider the following setup: $S = G/\bZ(G)$, where $\bG$ is a simple
simply connected algebraic group over the algebraic closure of a finite field
of characteristic~$p$ and $F:\bG\rightarrow\bG$ is a Steinberg endomorphism
with group of fixed points $G:=\bG^F$. We let $q$ denote the absolute value
of all eigenvalues of $F$ on the character group of an $F$-stable maximal
torus of $\bG$.

\begin{lem}   \label{defi}
 Theorems~\ref{simple} and~\ref{simple2} hold in the case $S$ is a simple
 group of Lie type in characteristic $r=p$.
\end{lem}

\begin{proof}
By Lemma~\ref{alt} we may assume that $(S,2) \neq (\tw2 F_4(2)',2)$. By
Lemma~\ref{center}(i) we may replace $S$ by $G = \bG^F$. Set
$s:=2$ in the case of Theorem \ref{simple2}.
By \cite[Prop.~5.1.7]{C}, $G$ contains a regular unipotent
$p$-element $g \in G$. Since $s \nmid |g^G|$, $\bC_G(g)$ contains a Sylow
$s$-subgroup $Q$ of $G$. But every semisimple element in $\bC_\bG(g)$ belongs
to $\bZ(\bG)$ by \cite[Prop.~5.1.5]{C}, whence $Q \leq \bZ(G)$. It follows
that $s \nmid |S|$ and so we are done.
\end{proof}

\begin{proof}[Proof of Theorem \ref{simple2}]
Assume the contrary: $|S:\cent Sx|$ is odd for all $r$-elements $x \in S$.
By Lemmas~\ref{alt} and~\ref{defi} we see that $S$ is a simple group of Lie
type in characteristic $p$ for some prime $p \neq r$ and
$(S,p) \neq (\tw2 F_4(2)',2)$.

Note that if $1 \neq g \in S$ is a {\bf real} $r$-element, then $|\bN_S(\langle
g \rangle):\bC_S(g)|$ is even and so $|g^S|$ is even. Thus $S$ cannot contain
any real $r$-element $g \neq 1$. By \cite[Prop.~3.1]{TZ}, it follows that
$$S \in \{ \PSL_n(q), \PSU_n(q) \mid n \geq 3\} \cup \{P\Omega^{\pm}_{4n+2}(q)
     \mid n \geq 2\} \cup \{E_6(q), \tw2 E_6(q)\}.$$
By Lemma \ref{center}(i) we may replace $S$ by $G = \bG^F$, or by some
quotient $L = G/Z$ with $Z \leq \bZ(G)$.

\smallskip
Suppose first that $S=\SL_n(q)$ with $n\geq3$, and set $k:=\ord_r(q)\leq n$.
If $k \leq n/2$, then $\SL_n(q) \geq \Sp_{2k}(q)$ contains a nontrivial real
$r$-element by \cite[Prop.~3.1]{TZ}, a contradiction. Similarly, if $2|k$,
then again $\SL_n(q) \geq \Sp_k(q)$ contains a nontrivial real $r$-element.
Thus $k > n/2$ and $k$ is odd.
Now it is easy to see that $H := \SL_k(q)$ contains an $r$-element $g$ with
$\bC_H(g) \cong C_{(q^k-1)/(q-1)}$. Embedding $H$ naturally in $\SL_n(q)$, we
get that
$$|g^G| = \frac{|\GL_n(q)}{(q^k-1) \cdot |\GL_{n-k}(q)|}
        = q^{\binom{n}{2}-\binom{n-k}{2}} \cdot \frac{\prod^n_{j=n-k+1}(q^j-1)}{q^k-1}$$
is even.

The same argument as above applies to the case $G = \SU_n(q)$ if we replace
$q$ by $-q$.

\smallskip
Suppose now that $S = P\Omega^\eps_{4n+2}(q)$ with $n \geq 2$.
Then we replace $S$ by $L = \Omega^\eps_{4n+2}(q)$. If
$r|\prod^{2n}_{j=1}(q^{2j}-1)$, then $\Omega^\eps_{4n+2}(q) > \Omega_{4n+1}(q)$
contains a nontrivial real $r$-element by \cite[Prop.~3.1]{TZ}, a
contradiction. So $r \nmid \prod^{2n}_{j=1}(q^{2j}-1)$ but $r|(q^{2n+1}-\eps)$.
Now it is easy to see that $H:= \SO^\eps_{4n+2}(q)$ contains an $r$-element
$g \in \Omega^\eps_{4n+2}(q)$ with $\bC_H(g) \cong C_{q^{2n+1}-\eps}$.
It follows that $|g^L|$ is even.

\smallskip
Finally, let $G := E^\eps_{6}(q)_{sc}$ with $\eps = +$
for $E_6(q)$ and $\eps = -$ for $\tw2 E_6(q)$. If $r$ divides $|F_4(q)|$,
then $G > F_{4}(q)$ contains
a nontrivial real $r$-element by \cite[Prop.~3.1]{TZ}, a contradiction.
So $r \nmid |F_4(q)|$ but $r|(q^5-\eps)(q^9-\eps)$. In particular, $r$ is a
Zsigmondy prime divisor for $q^5-\eps$ or $q^9-\eps$. Inspecting the
centralizers of semisimple elements in $G$ of order divisible by $r$,
as described in \cite{D},
one sees that there exist $r$-elements $h$ with $|C_G(h)|$ even.
\end{proof}

We will now prove the following result, which,
together with Lemmas \ref{alt}, \ref{defi} and Theorem \ref{simple2},
implies Theorem \ref{simple}.

\begin{thm}   \label{simple1}
 Let $S$ be a finite non-abelian simple group of Lie type defined over $\FF_q$,
 $q$ a power of a prime $p$, and let $r$ and $s$ be distinct odd prime
 divisors of $|S|$ different from $p$. Suppose that all the $r$-elements of
 $S$ have $s'$-conjugacy class size and all the $s$-elements of $S$ have
 $r'$-conjugacy class size. Then $S$ has an abelian Hall $\{r,s\}$-subgroup.
\end{thm}

\begin{prop}   \label{sl}
 Theorem \ref{simple1} holds for $S = \PSL_n(q)$.
\end{prop}

\begin{proof}
By Lemma \ref{center}(i) we may replace $S$ by $G = \SL_n(q)$.
Let $k := \ord_r(q)$ and $l := \ord_s(q)$; also let
$(q^k-1)_r =: r^t$. Since $r$ and $s$ divide $|S|$, $n \geq k,l$.
By symmetry we may assume $k \geq l$.

\smallskip
(a) Suppose first that $k \geq 2$, and so $r \nmid (q-1)$. Let $m \in \ZZ$ be
such that $r^m \leq n/k < r^{m+1}$ and denote
$\kappa := kr^m$. Since $r > 2$, $(q^\kappa-1)_r = r^{m+t}$. Furthermore, since
$r \nmid (q-1)$, the Singer cycle $\GL_1(q^{\kappa}) \leq \GL_{\kappa}(q)$
contains an $r$-element $h \in \SL_\kappa(q)$ of order $r^{m+t}$. Then for the
$r$-element $g := \diag(h,I_{n-\kappa}) \in G$ we have that
$$|g^G| = \frac{|\GL_n(q)|}{(q^\kappa-1) \cdot |\GL_{n-\kappa}(q)|}
        = \frac{|\GL_n(q)|}{|\GL_\kappa(q)| \cdot |\GL_{n-\kappa}(q)|} \cdot
                \frac{|\GL_\kappa(q)|}{q^\kappa-1}$$
which is divisible by
$$\frac{|\GL_\kappa(q)|_{p'}}{q^\kappa-1} = \prod^{\kappa-1}_{j=1}(q^j-1).$$
Since $s \nmid |g^G|$, we must have that $l \geq \kappa = kr^m \geq k \geq l$.
It follows that $k=l$, $m = 0$, i.e. $n/k < r$.

\smallskip
(b) Next we consider the case $k=l=1$, i.e. $r,s |(q-1)$.

Suppose first that $n \geq r+1$. Then we can find
$\alpha \in \overline{\FF}_q^\times$ of order $(q^r-1)_r = r^{t+1}$ and
consider the $r$-element $g\in G$ conjugate (in $\bG:=\SL_n(\overline{\FF}_q)$)
to
$$\diag(\alpha,\alpha^q, \ldots ,\alpha^{q^{r-1}},\alpha^{-\frac{q^r-1}{q-1}},
    \underbrace{1, \ldots ,1}_{n-r-1}).$$
Note that $1 \neq \alpha^{-\frac{q^r-1}{q-1}} \in \FF_q^\times \not\ni \alpha$.
It follows that
$$|g^G| =  \frac{|\SL_n(q)|}{(q^r-1) \cdot |\GL_{n-r-1}(q)|}
        =  \frac{|\GL_n(q)|}{|\GL_{r+1}(q)| \cdot |\GL_{n-r-1}(q)|} \cdot
                \frac{|\GL_{r+1}(q)|}{(q^r-1)(q-1)}$$
which is divisible by
$$\prod^{r-1}_{j=2}(q^j-1) \cdot (q^{r+1}-1),$$
a multiple of $s$, a contradiction.

We have shown that $n \leq r$, and so $n \leq s$ as well by symmetry. Assume
now that $n = r < s$. Then we can find $\beta \in \FF_q^\times$ of
order $s$ and consider the $s$-element
$$h := \diag(\beta,\beta^{-1}, \underbrace{1, \ldots ,1}_{r-2}) \in G.$$
Then
$$|h^G| =  \frac{|\SL_r(q)|}{(q-1) \cdot |\GL_{r-2}(q)|}
        =  q^{2r-3} \cdot \frac{q^r-1}{q-1} \cdot \frac{q^{r-1}-1}{q-1}$$
which is divisible by $r$, a contradiction.

\smallskip
(c) We have shown that $k =l$ and $n/k < \min(r,s)$. In this case, for
$c := \lfloor n/k \rfloor$ we have that
$$\GL_n(q) \geq \GL_k(q)^c \geq \GL_1(q^k)^c$$
contains an abelian Hall $\{r,s\}$-subgroup, and so does $G$.
\end{proof}

\begin{prop}   \label{su}
 Theorem \ref{simple1} holds for $S = \PSU_n(q)$.
\end{prop}

\begin{proof}
By Lemma \ref{center} we may replace $S$ by $G = \SU_n(q)$.
Let $k := \ord_r(-q)$ and $l := \ord_s(-q)$; also let
$(q^k-(-1)^k)_r =: r^t$. Since $r$, $s$ divide $|S|$, we have that $n \geq k,l$.
By symmetry we may assume $k \geq l$. For any monic irreducible polynomial
$f(x) \in \FF_{q^2}[x]$ with a nonzero root $\gamma$, let $\fv$ denote the
(unique) monic irreducible polynomial over $\FF_{q^2}$ with $\gamma^{-q}$ as
a root.

\smallskip
(a) First we consider the case $k$ is odd.

\smallskip
(a1) Suppose that $k \geq 3$ is odd, and so $r \nmid (q+1)$. Choose $m\in\ZZ$
such that $r^m \leq n/k < r^{m+1}$ and denote $\kappa := kr^m$. Since $r > 2$,
$(q^\kappa+1)_r = r^{m+t}$. Furthermore, since $r \nmid (q+1)$, the cyclic
subgroup $\GU_1(q^{\kappa}) \leq \GU_{\kappa}(q)$ contains an $r$-element
$h \in \SU_\kappa(q)$ of order $r^{m+t}$, with an eigenvalue
$\theta \in \overline{\FF}_q^\times$ of order $r^{m+t}$. One can check that
the minimal polynomial $f(x) \in \FF_{q^2}[x]$ for $\theta$ over $\FF_{q^2}$
has degree $\kappa$, and $f = \fv$. Now for the $r$-element
$g := \diag(h,I_{n-\kappa}) \in G$ we have that
$$|g^G| = \frac{|\GU_n(q)|}{(q^\kappa+1) \cdot |\GU_{n-\kappa}(q)|}
        = \frac{|\GU_n(q)|}{|\GU_\kappa(q)| \cdot |\GU_{n-\kappa}(q)|} \cdot
                \frac{|\GU_\kappa(q)|}{q^\kappa+1}$$
which is divisible by
$$\frac{|\GU_\kappa(q)|_{p'}}{q^\kappa+1}=\prod^{\kappa-1}_{j=1}(q^j-(-1)^j).$$
Since $s \nmid |g^G|$, we must have that $l \geq \kappa = kr^m \geq k \geq l$.
It follows that $k=l$, $m = 0$, i.e. $n/k < r$.

\smallskip
(a2) Next we consider the case $k=l=1$, i.e. $r,s |(q+1)$.

Suppose first that $n \geq r+1$. Then we can find
$\alpha\in\overline{\FF}_q^\times$ of order $(q^r+1)_r = r^{t+1}$ and consider
the $r$-element $g \in G$ conjugate (in $\bG := \SL_n(\overline{\FF}_q)$) to
$$\diag(\alpha,\alpha^{-q},\ldots,\alpha^{(-q)^{r-1}},\alpha^{-\frac{q^r+1}{q+1}},
    \underbrace{1, \ldots ,1}_{n-r-1}).$$
Note that $\alpha^{-\frac{q^r-1}{q-1}}$ has order $r^t|(q+1)$, and $f = \fv$
for the minimal polynomial $f(x)$ of $\alpha$ over $\FF_{q^2}$ (which has odd
degree $r$). Now one can check that
$$|g^G| =  \frac{|\SU_n(q)|}{(q^r+1) \cdot |\GU_{n-r-1}(q)|}
        =  \frac{|\GU_n(q)|}{|\GU_{r+1}(q)| \cdot |\GU_{n-r-1}(q)|} \cdot
                \frac{|\GU_{r+1}(q)|}{(q^r+1)(q+1)}$$
which is divisible by
$$\prod^{r-1}_{j=2}(q^j-(-1)^j) \cdot (q^{r+1}-1),$$
a multiple of $s$, a contradiction.

We have shown that $n \leq r$, and so $n \leq s$ as well by symmetry. Assume
now that $n = r < s$. Then we can find $\beta \in \FF_{q^2}^\times$ of order
$s|(q+1)$ and consider the $s$-element
$$h := \diag(\beta,\beta^{-1}, \underbrace{1, \ldots ,1}_{r-2}) \in G.$$
Then
$$|h^G| =  \frac{|\SU_r(q)|}{(q+1) \cdot |\GU_{r-2}(q)|}
        =  q^{2r-3} \cdot \frac{q^r+1}{q+1} \cdot \frac{q^{r-1}-1}{q+1}$$
which is divisible by $r$, a contradiction.

\smallskip
(a3) We have shown that $k =l$ and $n/k < \min(r,s)$. In this case, for
$c := \lfloor n/k \rfloor$ we have that
$$\GU_n(q) \geq \GU_k(q)^c \geq \GU_1(q^k)^c$$
contains an abelian Hall $\{r,s\}$-subgroup, and so does $G$.

\smallskip
(b) Now we consider the case $k$ is even: $k = 2k_1$; in particular,
$r \nmid (q+1)$. Choose $m \in \ZZ$ such that $r^m \leq n/k_1 < r^{m+1}$
and denote $\kappa := k_1r^m$. Since $r > 2$, $(q^{2\kappa}-1)_r = r^{m+t}$.
Furthermore, since $r \nmid (q+1)$, the cyclic subgroup
$$\GL_1(q^{2\kappa}) \leq \GL_\kappa(q^2) \leq \GU_{2\kappa}(q)$$
contains an $r$-element $h \in \SU_{2\kappa}(q)$ of order $r^{m+t}$, with an
eigenvalue $\theta \in \overline{\FF}_q^\times$ of order $r^{m+t}$. One can
check that the minimal polynomial $f(x) \in \FF_{q^2}[x]$ for $\theta$ over
$\FF_{q^2}$ has degree $\kappa$, and $f \neq \fv$.

Now for the $r$-element $g := \diag(h,I_{n-2\kappa}) \in G$ we have that
$$|g^G| = \frac{|\GU_n(q)|}{(q^{2\kappa}-1) \cdot |\GU_{n-2\kappa}(q)|}
        = \frac{|\GU_n(q)|}{|\GU_{2\kappa}(q)| \cdot |\GU_{n-2\kappa}(q)|}\cdot
                \frac{|\GU_{2\kappa}(q)|}{q^{2\kappa}-1}$$
which is divisible by
$$\frac{|\GU_{2\kappa}(q)|_{p'}}{q^{2\kappa}-1}
   = \prod^{2\kappa-1}_{j=1}(q^j-(-1)^j).$$
Since $s \nmid |g^G|$, we must have that $l \geq 2\kappa = kr^m \geq k \geq l$.
It follows that $k=l$, $m = 0$, i.e. $n/k < r$. Hence, by symmetry,
$n/k < \min(r,s)$. Now for $c := \lfloor n/k \rfloor$ we have that
$$\GU_n(q) \geq \GU_k(q)^c > \GL_{k_1}(q^2)^c \geq \GL_1(q^k)^c$$
contains an abelian Hall $\{r,s\}$-subgroup, and so does $G$.
\end{proof}

\begin{prop}   \label{clas}
 Theorem \ref{simple1} holds for $S$ a simple symplectic or orthogonal group.
\end{prop}

\begin{proof}
By Lemma~\ref{center}(i) we may replace $S=\PSp_{2n}(q)$ by $G := \Sp_{2n}(q)$;
set $d := 2n$, and $\Cl = \Cl^+ = \Cl^- := \Sp$ in this case.
If $S = P\Omega^{(\eps)}_d(q)$, then, since both $r$ and $s$ are odd,
we may replace $S$ by $G := \SO^{(\eps)}_d(q)$; here $\eps = \pm$, and
we set $d = 2n$ or $2n+1$, and $\Cl:=\SO$.
Let $k := \ord_r(q)$ and $l := \ord_s(q)$; also let $(q^k-1)_r =: r^t$.
Since $r$ divides $|S|$, we have that $2n \geq k$; moreover,
$n \geq k$ if $k$ is odd. The same conditions hold for $l$. By symmetry we may
assume $k \geq l$. For any monic irreducible polynomial $f(x) \in \FF_q[x]$
with a nonzero root $\gamma$, let $\fv$ denote the (unique) monic irreducible
polynomial over $\FF_q$ with $\gamma^{-1}$ as a root.

\smallskip
(a) First we consider the case $k$ is odd; in particular,
$k \leq n$. Let $m \in \ZZ$ be such that
$r^m \leq n/k < r^{m+1}$ and denote
$\kappa := kr^m$. Since $r > 2$, $(q^\kappa-1)_r = r^{m+t}$.

\smallskip
(a1) Assume that $k= l$ and $n/k<\min(r,s)$. Setting $c:=\lfloor n/k \rfloor$,
we get chains of subgroups of index coprime to $rs$
$$\Sp_{2n}(q) \geq \Sp_{2kc}(q) > \GL_{kc}(q) \geq \GL_k(q)^c \geq \GL_1(q^k)^c,$$
$$\SO_{2n+1}(q) \geq \SO^+_{2n}(q) \geq \SO^+_{2kc}(q) > \GL_{kc}(q)
                \geq \GL_k(q)^c \geq \GL_1(q^k)^c,$$
$$\SO^-_{2n}(q) \geq \SO^+_{2n-2}(q) \geq \SO^+_{2k(c-1)}(q) > \GL_{k(c-1)}(q)
                \geq \GL_k(q)^{c-1} \geq \GL_1(q^k)^{c-1},$$
and so $G$ contains an abelian Hall $\{r,s\}$-subgroup.
So now we will assume that we are not in this case.

\smallskip
(a2) Suppose that $G \neq \SO^-_{2\kappa}(q)$, and so $G$ contains a natural
subgroup $\Cl^+_{2\kappa}(q)$. Then the cyclic subgroup
$$\GL_1(q^{\kappa}) \leq \GL_{\kappa}(q) < \Cl^+_{2\kappa}(q)$$
contains an $r$-element $h$ of order $r^{m+t}$, with an eigenvalue
$\theta \in \overline{\FF}_q^\times$ of order $r^{m+t}$. One can check that
the minimal polynomial $f(x) \in \FF_q[x]$ for $\theta$ over $\FF_q$ has
degree $\kappa$, and $f \neq \fv$. Now for the $r$-element
$g := \diag(h,I_{d-2\kappa}) \in G$ we have that
$$|g^G| = \frac{|\Cl^\eps_d(q)|}{(q^\kappa-1) \cdot |\Cl^\eps_{d-2\kappa}(q)|}
        = \frac{|\Cl^\eps_d(q)|}{|\Cl^+_{2\kappa}(q)| \cdot |\Cl^\eps_{d-2\kappa}(q)|} \cdot
                \frac{|\Cl^+_{2\kappa}(q)|}{q^\kappa-1}$$
which is divisible by
$$\frac{|\Cl^+_{2\kappa}(q)|}{q^\kappa-1},$$
a multiple of $\prod^{\kappa-1}_{j=1}(q^{2j}-1)$.
Since $s \nmid |g^G|$, we must have that $l \geq \kappa = kr^m \geq l$.
It follows that $k=l$, $m = 0$, i.e. $n/k < r$.

As the roles of $r$ and $s$ are symmetric, repeating the above argument but
with $r$ replaced by $s$ in the case $G \neq \SO^-_{2ks^i}(q)$ for $i \geq 0$,
we see that $n/k < s$ and arrive at (a1).

\smallskip
(a3) It remains to consider the case $G = \SO^-_{2\kappa}(q)$ (so that
$n = \kappa = kr^m$). Since $r \nmid (q^k+1)$, we have
that $m \geq 1$. Set $\kappa_1 := kr^{m-1}$. Then the
cyclic subgroup
$$\GL_1(q^{\kappa_1}) \leq \GL_{\kappa_1}(q) < \SO^+_{2\kappa_1}(q)$$
contains an $r$-element $h_1$ of order $r^{m+t-1}$, with an eigenvalue
$\theta_1 \in \overline{\FF}_q^\times$ of order $r^{m+t-1}$. One can check
that the minimal polynomial $f_1(x) \in \FF_q[x]$ for $\theta'$ over
$\FF_q$ has degree $\kappa_1$, and $f_1 \neq \fv_1$. Now for the
$r$-element $g_1 := \diag(h_1,I_{2n-2\kappa_1}) \in G$ we have that
$$|(g_1)^G| = \frac{|\SO^-_{2n}(q)|}{(q^{\kappa_1}-1) \cdot |\SO^-_{2n-2\kappa_1}(q)|}
    = \frac{|\SO^-_{2n}(q)|}{|\SO^+_{2\kappa_1}(q)| \cdot |\SO^-_{2n-2\kappa_1}(q)|} \cdot
                \frac{|\SO^+_{2\kappa_1}(q)|}{q^{\kappa_1}-1}$$
which is divisible by
$$\frac{|\SO^+_{2\kappa_1}(q)|_{p'}}{q^{\kappa_1}-1}
    = \prod^{\kappa_1-1}_{j=1}(q^{2j}-1).$$
Since $s \nmid |(g_1)^G|$, we must have that $l\geq\kappa_1 = kr^{m-1} \geq l$.
It follows that $k=l$, $m = 1$, $n = kr$, $G = \SO^-_{2kr}(q)$.

Now the roles of $r$ and $s$ are symmetric, and $G \neq \SO^-_{2ks^i}(q)$ for
any $i \geq 0$. So, repeating the argument in (a2) but with $r$ replaced by
$s$, we see that $r = n/k < s$. In this case, as shown in (a1), the subgroup
$\SO_{2kr-1}(q)$ of index coprime to $rs$ in $G$ contains an abelian Hall
$\{r,s\}$-subgroup $H$, and so $H$ is also an abelian Hall $\{r,s\}$-subgroup
of $G$.

\smallskip
(b) Now we consider the case $k$ is even; in particular, $n \geq k_1 := k/2$.
Let $m \in \ZZ$ be such that $r^m \leq n/k_1 < r^{m+1}$ and denote
$\kappa := k_1r^m$. By the choice of $k$, $r \nmid (q^{k_1}-1)$, whence
$r|(q^{k_1}+1)$. Since $r > 2$, $(q^\kappa+1)_r = r^{m+t}$.

\smallskip
(b1) Assume that $k = l$ and $n/k_1 < \min(r,s)$. Setting $c := \lfloor n/k_1 \rfloor$, we get chains of subgroups of index coprime to $rs$
$$\Sp_{2n}(q) \geq \Sp_{2k_1c}(q) \geq \Sp_{2k_1}(q)^c \geq \Sp_2(q^{k_1})^c
              \geq \GU_1(q^{k_1})^c,$$
$$\SO_{2n+1}(q) \geq \SO^{(-1)^c}_{2n}(q) \geq
  \SO^{(-1)^c}_{2k_1c}(q) \geq \SO^-_{2k_1}(q)^c
                > \GU_1(q^{k_1})^c,$$
$$\SO^{(-1)^{c-1}}_{2n}(q)>\SO_{2n-1}(q) \geq \SO^{(-1)^{c-1}}_{2k_1(c-1)}(q)
               \geq \SO^-_{2k_1}(q)^{c-1} \geq \GU_1(q^{k_1})^{c-1},$$
and so $G$ contains an abelian Hall $\{r,s\}$-subgroup.
So we will assume that we are not in this case.

\smallskip
(b2) Suppose that $G \neq \SO^+_{2\kappa}(q)$, and so $G$ contains a natural
subgroup $\Cl^-_{2\kappa}(q)$. Then the cyclic subgroup
$$\GU_1(q^{\kappa}) < \Cl^-_{2\kappa}(q)$$
contains an $r$-element $h$ of order $r^{m+t}$, with an eigenvalue
$\theta \in \overline{\FF}_q^\times$ of order $r^{m+t}$. One can check that the
minimal polynomial $f(x) \in \FF_q[x]$ for $\theta$ over $\FF_q$ has degree
$2\kappa$, and
$f = \fv$. Now for the
$r$-element $g := \diag(h,I_{d-2\kappa}) \in G$ we have that
$$|g^G| = \frac{|\Cl^\eps_d(q)|}{(q^\kappa+1)\cdot |\Cl^{-\eps}_{d-2\kappa}(q)|}
        = \frac{|\Cl^\eps_d(q)|}{|\Cl^-_{2\kappa}(q)| \cdot |\Cl^{-\eps}_{d-2\kappa}(q)|} \cdot
                \frac{|\Cl^-_{2\kappa}(q)|}{q^\kappa+1}$$
which is divisible by
$$\frac{|\Cl^-_{2\kappa}(q)|}{q^\kappa+1},$$
a multiple of $\prod^{\kappa-1}_{j=1}(q^{2j}-1)$.
Since $s \nmid |g^G|$, we must have that $l \geq \kappa = k_1r^m$.
If $m \geq 1$ in addition, then $l \geq k_1r \geq 3k_1 > k$, a contradiction.
So $m = 0$, $n/k_1 < r$, and $k/2 =k_1 \leq l \leq k$.

Now if $2|l < k$, then $s$ divides $q^{2j}-1$ with $1 \leq j := l/2 < \kappa$,
again a contradiction. So when $2|l$, we must have that $k=l$. In this case,
the roles of $r$ and $s$ are symmetric. Therefore, in the case
$G \neq \SO^+_{2ks^i}(q)$ for $i \geq 0$, repeating the above argument but
with $r$ replaced by $s$, we see that $n/k < s$ and arrive at (b1).

\smallskip
(b3) Here we consider the case where $G \neq \SO^+_{2\kappa}(q)$ but $l$ is
odd. As shown in (b2), we now have $n/k_1 < r$, $k/2 = k_1 \leq l < k$.
Write $n = a k_1 +b$ with $1 \leq a < r$ and $0 \leq b < k_1$. We again
consider the element $h \in \GU_1(q^\kappa)$ as in (b2).

Assume in addition that $G \neq \SO^{-\alpha}_{2a\kappa}(q)$ with
$\alpha := (-1)^a$, and so $G$ contains a natural subgroup
$\Cl^{(-1)^a}_{2a\kappa}(q) > \GU_{1}(q^\kappa)^a$. Now for the $r$-element
$$g_1:= \diag(\underbrace{h, \ldots, h}_{a},I_{d-2a\kappa}) \in G$$
we have that
$$|(g_1)^G| = \frac{|\Cl^\eps_d(q)|}{|\GU_{a}(q^\kappa) \cdot |\Cl^{\alpha\eps}_{d-2a\kappa}(q)|} =
                \frac{|\Cl^\eps_d(q)|}{|\Cl^{\alpha}_{2a\kappa}(q)| \cdot |\Cl^{\alpha\eps}_{d-2a\kappa}(q)|} \cdot
                \frac{|\Cl^\alpha_{2a\kappa}(q)|}{|\GU_a(q^\kappa)|}$$
which is divisible by
$$\frac{|\Cl^\alpha_{2a\kappa}(q)|}{|\GU_a(q^\kappa)|}.$$
Since $s \nmid |(g_1)^G|$, we must have that
\begin{equation}\label{case1}
  s \nmid \prod^{ak_1-1}_{j=1,~k_1 \nmid j}(q^{2j}-1) \cdot \prod^{a-1}_{i=1}(q^{ik_1}+(-1)^i).
\end{equation}
In particular, if $a \geq 2$, then $s \nmid (q^{k_1}-1)$ and so $l > k_1$.
As $l \leq k$ and $l$ is odd, it also follows that $k_1 \nmid l$.
In this case, $s |(q^{2j}-1)$ with
$1 \leq j := l \leq k-1 \leq ak_1-1$ and $k_1 \nmid j$, contrary to
(\ref{case1}). So $a = 1$ and $k_1 \leq n < 2k_1 = k$. Also, recall that $l$
is odd and $k_1 \leq l \leq n$. As $s$ divides $|S|$, we must have
that $G \neq \SO^-_{2l}(q)$, and so $G$ contains a natural
subgroup $\Cl^+_{2l}(q)$. Then the cyclic subgroup
$$\GL_1(q^{l}) < \Cl^+_{2l}(q)$$
contains an $s$-element $h_2$ of order $(q^l-1)_s$, with an eigenvalue
$\theta_2 \in \overline{\FF}_q^\times$ of order $(q^l-1)_s$. One can check
that the minimal polynomial $f_2(x) \in \FF_q[x]$ for $\theta_2$ over $\FF_q$
has degree $l$, and $f_2 \neq \fv_2$. Now for the $s$-element
$g_2 := \diag(h_2,I_{d-2l}) \in G$ we have that
\begin{equation}\label{caseb31}
  |g_2^G| = \frac{|\Cl^\eps_d(q)|}{(q^l-1)\cdot |\Cl^{\eps}_{d-2l}(q)|}
        = \frac{|\Cl^\eps_d(q)|}{|\Cl^+_{2l}(q)| \cdot |\Cl^{\eps}_{d-2l}(q)|} \cdot
                \frac{|\Cl^+_{2l}(q)|}{q^l-1}
\end{equation}
which is divisible by
$$\frac{|\Cl^+_{2l}(q)|}{q^l-1},$$
a multiple of $\prod^{l-1}_{j=1}(q^{2j}-1)$.
Since $r \nmid |g_2^G|$, we must have that $k_1 \geq l$, whence in fact
$k_1 = l$. But in this case, $r \nmid (q^l-1) \cdot |\Cl^\eps_{d-2l}(q)|$,
and so (\ref{caseb31}) implies that $r$ divides $|g_2^G|$, a contradiction.

Now suppose that $G = \SO^{-\alpha}_{2a\kappa}(q)$. Since $r||G|$, we must
have that $a \geq 2$. Now for the $r$-element
$$g_3 := \diag(\underbrace{h, \ldots, h}_{a-1},I_{2\kappa}) \in G$$
we have that
$$|(g_3)^G| = \frac{|\SO^{-\alpha}_{2a\kappa}(q)|}{|\GU_{a-1}(q^\kappa)
              \cdot |\SO^{+}_{2\kappa}(q)|},$$
and so
\begin{equation}\label{caseb32}
  s \nmid \frac{\prod^{ak_1-1}_{j=1,~k_1 \nmid j}(q^{2j}-1) \cdot
          \prod^{a}_{i=2}(q^{ik_1}+(-1)^i)}{\prod^{k_1-1}_{i=1}(q^{2i}-1)}.
\end{equation}
Since $l \geq k_1$ is odd, $s$ does not divide the denominator in
(\ref{caseb32}). On the other hand, if $k_1 \nmid l$, then $s$ divides
$q^{2l}-1$ and so divides the numerator of (\ref{caseb32}), a contradiction.
So $k_1 |l<k=2k_1$, i.e. $l = k_1$.
If $a \geq 3$, then $s$ divides $q^{3k_1}-1$ and so divides the numerator,
again a contradiction. We conclude that $a = 2$, $G = \SO^-_{4l}(q)$.
In this case,
$$\GL_1(q^l) \times \GU_1(q^l) <
  \SO^+_{2l}(q) \times \SO^-_{2l}(q) < \SO^-_{4l}(q)$$
contains an abelian Hall $\{r,s\}$-subgroup of $G$.

\smallskip
(b4) It remains to consider the case $G = \SO^+_{2\kappa}(q)$ (so that
$n = \kappa = k_1r^m$). Since $r \nmid (q^{k_1}-1)$, we have that
$m \geq 1$. Set $\kappa_1 := k_1r^{m-1}$. Then the cyclic subgroup
$$\GU_1(q^{\kappa_1}) < \SO^-_{2\kappa_1}(q)$$
contains an $r$-element $h_1$ of order $r^{m+t-1}$, with an eigenvalue
$\theta_1 \in \overline{\FF}_q^\times$ of order $r^{m+t-1}$. One can check
that the minimal polynomial $f_1(x) \in \FF_q[x]$ for $\theta_1$ over
$\FF_q$ has degree $2\kappa_1$, and $f_1 = \fv_1$. Now for the $r$-element
$$g_1 := \diag(\underbrace{h_1, \ldots, h_1}_{r-1},I_{2\kappa_1}) \in G$$
we have that
$$|g^G| = \frac{|\SO^+_{2\kappa_1 r}(q)|}{|\GU_{r-1}(q^{\kappa_1})| \cdot |\SO^+_{2\kappa_1}(q)|} =
                \frac{|\SO^+_{2\kappa_1 r}(q)|}{|\SO^+_{2\kappa_1(r-1)}(q)| \cdot |\SO^+_{2\kappa_1}(q)|} \cdot
                \frac{|\SO^+_{2\kappa_1(r-1)}(q)|}{|\GU_{r-1}(q^{\kappa_1})|}.$$
Since $s \nmid |g^G|$, we must have that
\begin{equation}\label{case2}
  s \nmid \frac{|\SO^+_{2\kappa_1(r-1)}(q)|_{p'}}{|\GU_{r-1}(q^{\kappa_1})|_{p'}} =
    \prod^{\kappa_1(r-1)-1}_{j=1,~\kappa_1 \nmid j}(q^{2j}-1) \cdot
    \prod^{r-2}_{i=1}(q^{i\kappa_1} + (-1)^i).
\end{equation}
Now if $\kappa_1 \nmid l$, then we have $s|(q^{2l}-1)$ with
$1 \leq l < k = 2k_1 \leq \kappa_1(r-1)$,
contrary to (\ref{case2}).  So $k_1r^{m-1} = \kappa_1|l \leq k = 2k_1$,
yielding $m = 1$ and $k_1|l$. If furthermore $l = k_1$ then again
$s|(q^{\kappa_1}-1)$, contradicting~(\ref{case2}). Thus $k=l$ and
$G = \SO^+_{2k_1r}(q)$.

Now the roles of $r$ and $s$ are symmetric, and $G \neq \SO^+_{2k_1s^i}(q)$
for any $i \geq 0$. So, repeating the argument in (b2) but with $r$ replaced
by $s$, we see that $r = n/k_1 < s$. In this case, as shown in (b1), the
subgroup $\SO_{2k_1r-1}(q)$ of index coprime to $rs$ in $G$ contains
an abelian Hall $\{r,s\}$-subgroup $H$, and so $H$ is also an abelian
Hall $\{r,s\}$-subgroup of $G$.
\end{proof}

We now turn to the case of exceptional simple groups of Lie type.

\begin{prop}   \label{prop:bad}
 Let $\bG$ be simple such that $G=\bG^F$ is of exceptional type. Let $r\ne2$
 be a bad prime for $\bG$, or $r=3$ when $G=\tw3D_4(q)$. Let $S\le G$ be
 a Sylow $r$-subgroup of $G$. Then $\bC_G(S)=\bZ(G)\bZ(S)$, unless $G=E_7(q)$
 and $r=3$, when $\bC_G(S)=\bZ(S)T$ with $T$ a torus of order $q-\eps$,
 where $q\equiv\eps\pmod3$.
\end{prop}

\begin{proof}
By assumption we have $r=3$, or $r=5$ when $\bG$ is of type $E_8$. For
almost all cases, the normalizer of a Sylow $r$-subgroup of $G$ is given
in \cite[Tab.~1]{MN}, and the claim follows. The only remaining case is when
$G=E_7(q)$ with $r=3$. Here, $S$ is contained in a Levi subgroup of type
$E_6$ or $\tw2E_6$, and the claim follows from the case of $E_6$ discussed
before.
\end{proof}

In what follows, we consider the cyclotomic polynomials, over $\QQ(\sqrt{p})$
in the case of Suzuki-Ree groups (see \cite[\S F]{BM}) and over $\QQ$
otherwise, that occur in the generic order of $G$.

\begin{prop}   \label{prop:weyl}
 Let $\bG$ be simple such that $G=\bG^F$ is of exceptional type. Let $r\ne2$
 be a good prime for $\bG$ which divides two distinct cyclotomic polynomials
 $\Phi_d(q)$ occurring in the generic order of $G$. Then the following
 holds for a Sylow $r$-subgroup of $G$:
 \begin{enumerate}
  \item[\rm(a)] $G=E_6(q)$, $r=5$, $|\bC_G(S)|$ divides $|\bZ(S)|q(q-1)^2$;
  \item[\rm(b)] $G=\tw2E_6(q)$, $r=5$, $|\bC_G(S)|$ divides $|\bZ(S)|q(q+1)^2$;
  \item[\rm(c)] $G=E_7(q)$, $r=5$, $|\bC_G(S)|$ divides $|\bZ(S)|q(q\pm1)^3$; or
  \item[\rm(d)] $G=E_8(q)$, $r=7$, $|\bC_G(S)|$ divides $|\bZ(S)|q(q\pm1)^2$.
 \end{enumerate}
\end{prop}

\begin{proof}
Note that the assumptions on $r$ already exclude the Suzuki-Ree groups.
Now, if $r$ divides two distinct cyclotomic polynomials occurring in the
generic order of $G$, then $r$ divides the order of the Weyl group of $G$,
see \cite[Cor.~3.13]{BM}. Since $r$
was also assumed to be good, the only possibilities are that $r=5$ and
$\bG$ is of type $E_6$ or $E_7$, or $r=7$ and $\bG$ is of type $E_8$. \par
First assume that $G=E_6(q)$. If $q\not\equiv1\pmod5$, then the Sylow
$5$-subgroups of $G$ are cyclic. Else, a Sylow 5-subgroup of $G$ is contained
in a Levi subgroup of type $A_4$, with centralizer of type $A_1$, whence
the claim in this case. Entirely similar arguments apply for $\tw2E_6(q)$. If
$G=E_7(q)$ with $r=5$, then our assumption gives $q\equiv\pm1\pmod5$ and
again a Sylow 5-subgroup of $G$ lies inside a Levi subgroup of type $A_4$,
centralized by a Levi subgroup of type $A_2$. Finally, in $E_8(q)$, a Sylow
7-subgroup $S$ does not lie in maximal tori if and only if $q\equiv\pm1\pmod7$.
In that case, $S$ is contained in a Levi subgroup of type $A_6$, with
centralizer of type $A_1$.
\end{proof}

Next, for Suzuki-Ree groups let $\Phi_{d_1}$ be the cyclotomic polynomial
occurring in the generic order of $G$ such that $r|\Phi_{d_1}(q)$ as in
\cite[App.~2]{BM}, and similarly $\Phi_{d_2}$ for $s$. For other groups,
let $d_1$ and $d_2$ be the order of $q$ modulo $r$, respectively $s$.

\begin{prop}   \label{prop:good}
 Assume that the Sylow $d_1$-tori of $\bG$ are maximal tori and that $s$
 divides a unique cyclotomic factor in the generic order of $G$. If every
 $s$-element of $G$ centralizes a Sylow $r$-subgroup of $G$,
 then $d_1=d_2$.
\end{prop}

\begin{proof}
Let $g$ be an $s$-element such that $\bC_G(g)$ contains a Sylow
$r$-subgroup. Then by \cite[Cor.~3.13]{BM},
$g$ lies in a maximal torus $\bT$ containing a Sylow
$d_1$-torus $\bT_1$ of $\bG$. Since $\bT_1$ is a maximal torus by assumption,
$\bT = \bT_1$. So $s$ divides $\Phi_{d_1}(q)$ which by our
assumption implies that $d_2=d_1$.
\end{proof}

\begin{prop}   \label{prop:d1=d2}
 Assume that $d_1=d_2$ and neither of $r,s$ divides the order of
 the Weyl group of $\bG$. Then there exists an abelian $\{r,s\}$-Hall
 subgroup of $G$.
\end{prop}

\begin{proof}
Since neither $r$ nor $s$ divides the order of the Weyl group of $\bG$,
any Sylow $d$-torus $\bT_d$ of $\bG$, where $d=d_1=d_2$, has the property
that $\bT_d^F$ contains a Sylow $r$-subgroup $S_1$ and a Sylow $s$-subgroup
$S_2$ of $G$ by \cite[Cor.~3.13]{BM}. In particular, $S_1 \times S_2$
is an abelian Hall $\{r,s\}$-subgroup of $G$.
\end{proof}

Now we can complete the proof of Theorem \ref{simple1}:

\begin{thm}  \label{thm:proof2.1}
 The assertion of Theorem~\ref{simple1} holds for $S$ a simple exceptional
 group of Lie type.
\end{thm}

\begin{proof}
As before, there exists a simple, simply connected algebraic group $\bG$
with a Steinberg endomorphism $F$ such that $S=G/\bZ(G)$, where $G:=\bG^F$.
By Lemma \ref{center}(i), we may replace $S$ by $G$.

First assume that $r$ is a bad prime for $G$. Then Proposition~\ref{prop:bad}
yields the contradiction that $s \nmid |S|$, unless $G=E_7(q)$, $r=3$ and
$s|(q-\eps)$, where $q\equiv\eps\pmod3$.
But $E_7(q)$ has a Levi subgroup of type $D_6$, with center a torus of order
$q-\eps$, and $s$-elements (which have order at least~5) in that center
do not centralize a Sylow 3-subgroup of $G$. \par
Thus, both primes $r,s$ are good for $G$. Next assume that $r$, say,
satisfies the assumptions of Proposition~\ref{prop:weyl}. Then necessarily
$s$ divides $q\pm1$. In all of the cases (a)--(d) we will exhibit a Levi
subgroup occurring as the centralizer of an $s$-element of order dividing
$q\pm1$ but not containing a Sylow $r$-subgroup. For $G=E_6(q)$, take a Levi
of type $A_2^2A_1$ (note that here $s\ge7$), and similarly for $\tw2E_6(q)$.
For $E_7(q)$, we may take a Levi of type $A_3A_2A_1$, and for $E_8(q)$ with
$r=7$ a Levi of type $E_6A_1$ will do. \par
We may now assume that each of $r$ and $s$ divides a unique cyclotomic
polynomial $\Phi_{d_i}(q)$ occurring in the generic order of $G$, but
does not divide the order of the Weyl group of $\bG$.
If $d_1 = d_2$, we are done by Proposition \ref{prop:d1=d2}.
Otherwise, by Proposition \ref{prop:good} we have that neither of the Sylow
$d_i$-tori are maximal tori of $\bG$. Then $d_1$ and $d_2$ are as in
Table~\ref{tab:exc}. In these cases, the last column of the table gives
certain centralizers of $r$-elements which do not contain a Sylow
$s$-subgroup, a contradiction.
\end{proof}

\begin{table}[htbp]
 \caption{Non-maximal non-cyclic Sylow tori}
  \label{tab:exc}
\[\begin{array}{|cll|}
\hline
 G& d_i\\
\hline \hline
 \tw3D_4(q)& 1,2& (q-1).A_1(q^3)\\
     E_6(q)& 2,4,6& (q^2-1).\tw2D_4,(q^2+1)(q-1).\tw2A_3(q)\\
 \tw2E_6(q)& 1,3,4& (q^2-1).\tw2D_4,(q^2+1)(q+1).A_3(q)\\
     E_7(q)& 3,4,6& (q^3\pm1).\tw3D_4(q)\\
\hline
\end{array}\]
\end{table}

\section{Proof of Theorems A, B and C}

We start with a trivial observation.

\begin{lem}   \label{obs}
 Suppose that $N \nor G$, and let $p$ and $q$ be primes.
 If all the $p$-elements of $G$ have conjugacy class size not divisible by $q$,
 then the same happens in $G/N$ and $N$.
\end{lem}

\begin{proof}
If $x \in N$ is a $p$-element, then $|N:\cent Nx|$ divides $|G:\cent Gx|$,
which is not divisible by $q$, by hypothesis. If $Nx \in G/N$ is a $p$-element,
then $Nx=Ny$ where $y$ is the $p$-part of $x$. If $D/N=\cent{G/N}{Nx}$,
then $\cent Gy N/N \leq D/N$ and $|G:D|$ divides $|G:\cent Gy|$, which is not
divisible by $q$, by hypothesis.
\end{proof}

\begin{lem}   \label{lema}
 Let $q$ be a prime. Suppose that a $q$-group $Q$ acts coprimely on a finite
 group $N$. Let $p$ be a prime and let $P$ be a $Q$-invariant Sylow
 $p$-subgroup of $N$. Assume that for every $x \in Q$, there exists $n \in N$
 such that $[x, P^n]=1$. Then $[P,Q]=1$.
\end{lem}

\begin{proof}
We argue by induction on $|Q|$.
Suppose that $R$ that is a maximal subgroup of $Q$. By induction,
we have that $[R,P]=1$. Suppose that $S$ is another maximal subgroup of $Q$.
Then $[S,P]=1$, and therefore $[Q,P]=1$ since $RS=Q$. Hence,
we conclude that $Q$ has a unique maximal subgroup. Then $Q/\Phi(Q)$ is cyclic,
and therefore $Q=\langle x\rangle$. Now, by hypothesis, there is $n \in N$
such that $[Q,P^n]=1$. In particular, $P^n$ is $Q$-invariant. Since $P$
is $Q$-invariant, by Glauberman's Lemma, there is $c\in \cent NQ$ such that
$P=(P^n)^c \leq \cent NQ$, as desired.
\end{proof}

We will use the following consequence in several places below.

\begin{cor}   \label{cor}
 Suppose that $G$ is a finite group, and let $p,q$ be different primes.
 Assume that every $q$-element of $G$ has conjugacy class of size not divisible
 by $p$. Suppose that $N \nor G$ is a $q'$-group. If $Q \in \syl qG$ and
 $P \in \syl pN$ is $Q$-invariant, then $[Q,P]=1$.
\end{cor}

\begin{proof}
 We have that
$G=N\norm GP$ by the Frattini argument.
Now, let $x \in Q$. By hypothesis, there is $P_1 \in \syl pG$ such that
$[x,P_1]=1$. Now, $P^n \leq P_1$ for some $n \in N$, and thus $[x,P^n]=1$.
Now Lemma~\ref{lema} applies.
\end{proof}

Now, we are ready to prove Theorem A of the introduction. Recall that
if a finite group $G$ has a nilpotent Hall $\pi$-subgroup $H$, then every
$\pi$-subgroup of $G$ is contained in some $G$-conjugate of $H$ by a
well-known theorem of Wielandt.

\begin{thm}
 Suppose that $G$ is a finite group. Then $G$ has nilpotent $\{p,q\}$-Hall
 subgroups if and only if for every $p$-element $x \in G$, then $|G:\cent Gx|$
 is not divisible by $q$, and for every $q$-element $y \in G$, then
 $|G:\cent Gy|$ is not divisible by $p$.
\end{thm}

\begin{proof}
We assume the condition $^*$: for every $p$-element $x \in G$,
$|G:\cent Gx|$ is not divisible by $q$, and for every $q$-element $y \in G$,
$|G:\cent Gy|$ is not divisible by $q$.
We prove by induction on $|G|$ that $G$ has a nilpotent Hall
$\{p,q\}$-subgroup. Write $\pi=\{p,q\}$.

The condition $^*$ is inherited by quotients and normal subgroups,
by Lemma~\ref{obs}.

Let $1<N$ be a normal subgroup of $G$. By induction, we know that $G/N$ has
a nilpotent Hall $\pi$-subgroup $H/N$. Suppose that $|N|$ is not divisible by
$p$ or $q$. Then we use the Schur--Zassenhaus theorem in $H$ to get a
nilpotent Hall $\pi$-subgroup of $G$.

Suppose now that $|N|$ is not divisible by $p$.
Let $P \in \syl p G$ and let $Q \in \syl q N$ be $P$-invariant (which we know
to exist by coprime action).
By Corollary~\ref{cor}, we have that $[P,Q]=1$.
Now, recall that $G/N$ has a Hall nilpotent $\pi$-subgroup $H/N$. Thus,
using the
Frattini argument and the Schur-Zassenhaus theorem in the group $\norm HQ/Q$
with respect to the normal subgroup $\norm NQ/Q$, we have that
$\norm HQ/Q$ has a nilpotent Hall $\pi$-subgroup $U/Q$, which we may assume
contains $P$.
Now, notice that $U$ is a Hall $\pi$-subgroup of $G$. Write
$U/Q=(S/Q) \times (PQ/Q)$, where $S \in \syl q U$. (In particular,
$S \in \syl q G$.) Then $[P,S] \leq Q$ and $[S,P,P]=1$. Thus $[S,P]=1$ by
coprime action.

Hence we may assume that the order of every proper normal subgroup is
divisible by $p$ and $q$.

Let $N$ be a minimal normal subgroup of $G$. Hence
$N=S_1 \times \cdots \times S_k$, where $S_i$ is a non-abelian simple group
of order divisible by $pq$, and $G$ transitively permutes the set
$\Omega=\{ S_1, \ldots, S_k \}$. Let $B$ be the kernel of the action.
We claim that $G/B$ is a $q'$-group. Otherwise, let $Bx$ be an element of
order $q$, where $x$ has $q$-power order. Now, by hypothesis, we have that
$[x,P]=1$ for some Sylow $p$-subgroup $P$ of $G$. Then, $P\cap N \in \syl pN$,
and in fact $P\cap N=(P \cap S_1) \times \ldots \times (P \cap S_k)$.
Now, let $y \in P \cap S_i$ be of order $p$. Then
$[y,x]=1$, and therefore we deduce that $x$ has to normalize $S_i$.
Then $x \in \bigcap_i \norm G{S_i}=B$ and this is a contradiction.
This shows that $G/B$ is a $q'$-group and by symmetry a $p'$-group. But
then $B$ contains both Sylow $p$-subgroups and
$q$-subgroups of $G$ and therefore by induction, we may assume that $B=G$.
Thus $N$ is a simple group of order divisible by $pq$.

We show now that $G$ can be assumed to have no proper solvable quotients.
Suppose that $G/K$ has prime order, where $K \nor G$. By induction, we know
that $K$ has nilpotent Hall $\pi$-subgroups.
If $G/K$ is $\pi'$-group, then the nilpotent Hall $\pi$-subgroups
of $G$ are Hall subgroups of $G$ and we are done. Therefore, we assume (by symmetry) that $G/K$ has order $p$, so we may write
$G=K\langle x\rangle$ for some $p$-element $x\in G$.
Now, let $P \in \syl pG$ be such that $x \in P$. By hypothesis, let
$Q \in \syl qG$ such that $[Q,x]=1$. Notice that $Q\leq K$ and that
$P \cap K \in \syl pK$. Since $K$ has nilpotent Hall $\pi$-subgroups,
then $P\cap K$ is contained in some nilpotent Hall $\pi$-subgroup of $K$.
Hence there is $k \in K$ such that
$[Q^k, P\cap K]=1$. In particular, $|K:\cent K{Q^k}|$ is not divisible by $p$.
Since $x^k \in \cent G{Q^k}$, we have that $G=K\cent G{Q^k}$. Hence
$|G: \cent G{Q^k}|=|K:\cent K{Q^k}|$ is not divisible by $p$, and there is
some Sylow $p$-subgroup $P_1$ of $G$ such that $[Q^k, P_1]=1$. We deduce that
$G$ has nilpotent Hall $\pi$-subgroups, and in this case the theorem is proved.

Now, since $G/N\cent GN$ is isomorphic to a subgroup of ${\rm Out}(N)$, then
$G/N\cent GN$ is solvable and we conclude that $G=N \cent GN$. Since
$\cent GN$ has nilpotent Hall $\pi$-subgroups by induction, then we conclude
that $N$ cannot be proper in $G$, because otherwise $N$ and therefore $G$
would have nilpotent Hall $\pi$-subgroups. Now Theorem~\ref{simple} applies.
\end{proof}

Theorem B immediately follows from Theorem A, by using the following.

\begin{lem}
 Let $G$ be a finite group and let $\pi$ be a set of primes.
 Assume that $\pi$ contains at least two prime divisors of $|G|$.
 If G has nilpotent Hall $\tau$-subgroups for every $\tau \subseteq \pi$ with
 $|\tau|=2$, then $G$ has nilpotent Hall $\pi$-subgroups.
\end{lem}

\begin{proof}
This is \cite[Lemma~3.4]{M}. (See the comment that follows the proof.)
\end{proof}

Proposition 2.3 of \cite{KS} provides a different reduction to simple groups
of Theorem B. However, it is easier to check Theorem~\ref{simple} than
Theorem~B for simple groups.

Finally, Theorem C is now an obvious consequence of Theorem B and the Main
Theorem in \cite{NST}.

\section{Some Solvability}
Next, we show now that a version of Theorem A is possible under weaker
hypotheses if we allow some solvability conditions.

\begin{thm}   \label{t}
 Let $p,q$ be primes, and let $G$ be a finite group.
 Assume that all the $q$-elements have conjugacy class sizes not divisible
 by $p$. If $G$ is $p$-solvable or $q$-solvable, then a Sylow $p$-subgroup
 of $G$ normalizes some Sylow $q$-subgroup of $G$.
\end{thm}

\begin{proof}
We argue by induction on $|G|$. Assume first that $G$ is $p$-solvable.
Let $K=\oh{p'}G$ and let $L/K=\oh{p}{G/K}$.
Let $Kx$ be a $q$-element of $G/K$, where $x$ is a $q$-element of $G$.
Now, $[x,P]=1$ for some Sylow $p$-subgroup $P$ of $G$, by hypothesis.
Thus $[Kx,PK/K]=1$ and therefore $Kx$ centralizes $L/K$. By Hall-Higman
1.2.3. Lemma, it follows that $Kx \in L/K$, and thus $Kx=K$. We conclude
that $G/K$ is a $q'$-group. In particular, $K$ contains a Sylow $q$-subgroup
of $G$. Now, $P$ acts coprimely on $K$, and by coprime action, it follows that
$P$ normalizes some Sylow $q$-subgroup of $K$, which is a Sylow $q$-subgroup
of $G$.

Assume now that $G$ is $q$-solvable. If $1<N \nor G$, then by induction
we know that there exists $P \in \syl pG$ and $Q \in \syl qG$ such that
$P$ normalizes $NQ$. If $\oh qG>1$, then we set $N=\oh qG$, $P$ normalizes
$QN=Q$, and we are done. So we may assume that $\oh{q'}G=N>1$. Now $Q$ acts
coprimely on $N$. By coprime action $Q$ normalizes some $P_1 \in \syl pN$.
By Corollary~\ref{cor}, we have that $[Q,P_1]=1$, in particular
$|N:\norm N{Q}|$ is not divisible by $p$. Now $NQ \nor NQP$ and by the
Frattini argument, we have that $NQP=N\norm {NQP}Q$. Then
$|NQP:\norm {NQP}Q|=|N:\norm NQ|$ is not divisible by $p$. Hence some
Sylow $p$-subgroup of $NQP$ (and hence of $G$) normalizes $Q$.
\end{proof}

It is an interesting problem to study if the property that a Sylow $p$-subgroup
normalizes some Sylow $q$-subgroup is detectable by the character table.
(More generally, what does the character table of $G$ know about $\nu_q(G)$,
the number of Sylow $q$-subgroups of $G$?) We can easily solve this question
in $p$-solvable groups.

\begin{thm}
 Suppose that $G$ is $p$-solvable. Let $q\ne p$ be another prime.
 Then some Sylow $p$-subgroup of $G$ normalizes some Sylow
 $q$-subgroup of $G$ if and only if $G/\oh{p'} G$ is a $q'$-group.
\end{thm}

\begin{proof}
Suppose that $G/N$ is a $q'$-group, where $N=\oh{p'}G$.
Let $P \in \syl p G$. Then $P$ acts coprimely on $N$, and by coprime
action it stabilizes some Sylow $q$-subgroup of $N$, which is a Sylow
$q$-subgroup of $G$.

Conversely, suppose that $P \in \syl pG$ normalizes $Q \in \syl pG$.
We show by induction on $|G|$ that $G/\oh{p'}G$ is a $q'$-group.
If $N \nor G$, then we have that $PN/N$ normalizes $QN/N$, so by applying
induction in $G/\oh{p'}G$, we may assume that $\oh{p'}G=1$.
Now, let $K=\oh p G$. By hypothesis, we have that $K$ normalizes $Q$.
Also $Q$ normalizes $K$. Since $Q \cap K=1$, then we conclude that $[Q,K]=1$.
Then $Q \leq \cent GK \leq K$ by Hall-Higman's Lemma 1.2.3, and $Q=1$. This
concludes the proof.
\end{proof}

Finally, we prove that if the prime 2 is involved in the following form, then
we can obtain a certain solvability.

\begin{thm}
 Let $q$ be an odd prime, and let $G$ be a finite group.
 If all the $q$-elements of $G$ have conjugacy class size not divisible by $2$,
 then $G$ is $q$-solvable. In particular, a Sylow $2$-subgroup of $G$
 normalizes some Sylow $q$-subgroup of $G$.
\end{thm}

\begin{proof}
We argue by induction. If $1<N$ is a proper normal subgroup of $G$, then
$G/N$ and $N$ are $q$-solvable. So we may assume that $G$ is a non-abelian
simple group of order divisible by $2q$ and appeal to Theorem~\ref{simple2}.
The last part follows from Theorem~\ref{t}.
\end{proof}

We cannot reverse the primes in the previous theorem, not even to obtain
normalizing conditions between Sylow subgroups. All the 2-elements of
$G=J_1$ have conjugacy class not divisible by 5, and no Sylow 5-subgroup
of $G$ normalizes any Sylow 2-subgroup of $G$.


\end{document}